\documentclass[preprint,10pt]{elsarticle}

\usepackage{amssymb}
\usepackage{amsmath}
\usepackage[all,cmtip]{xy}
\usepackage{latexsym}
\usepackage{mathtools}
\usepackage{amsthm}
\usepackage{a4wide}
\usepackage{color}
\usepackage{hyperref}
\usepackage{tikz}
\usetikzlibrary{arrows}
\input diagxy

\journal{}

\theoremstyle{plain}
  \newtheorem{thm}{Theorem}[section]
  
  \newtheorem{prop}[thm]{Proposition}
  \newtheorem{cor}[thm]{Corollary}
\theoremstyle{definition}
  \newtheorem{defn}[thm]{Definition}
  
  \newtheorem{exmp}[thm]{Example}
  \newtheorem{rem}[thm]{Remark}

\def\rto{{\bfig\morphism<180,0>[\mkern-4mu`\mkern-4mu;]\place(82,0)[\mapstochar]\efig}}

\makeatletter
\def\ps@pprintTitle{%
 \let\@oddhead\@empty
 \let\@evenhead\@empty
 \def\@oddfoot{\centerline{\thepage}}%
 \let\@evenfoot\@oddfoot}
\makeatother

\newcommand{\Lra}{\Longrightarrow}

\newcommand{\bv}{\bigvee}
\newcommand{\bw}{\bigwedge}

\renewcommand{\phi}{\varphi}

\newcommand{\sk}{{\sf k}}

\newcommand{\sV}{{\mathcal V}}

\newcommand{\sy}{{\sf y}}

\newcommand{\Cat}{{\bf Cat}}

\newcommand{\Set}{{\bf Set}}

\newcommand{\VCat}{{\mathcal V}\text{-}\Cat}

\newcommand{\op}{{\rm op}}

\begin{document}

\begin{frontmatter}



\title{Metagories}


\author{Walter Tholen\fnref{A}}

\ead{tholen@mathstat.yorku.ca}

\author{Jiyu Wang\fnref{B}}

\ead{gateswjy@gmail.com}

\address{Department of Mathematics and Statistics, York University, Toronto, Ontario, Canada, M3J 1P3}
\address{\em Dedicated to Ale\v{s} Pultr on the occasion of his eightieth birthday}

\fntext[A]{The first author is supported by a Discovery Grant of the Natural Sciences and Engineering Research Council (NSERC) of Canada.}
\fntext[B]{The second author held a Dean's Undergraduate Research Award by the Faculty of Science of York University while undertaking research for this paper.}

\begin{abstract}
Metric approximate categories, or metagories, for short, are metrically enriched graphs. Their structure assigns to every directed triangle in the graph a value which may be interpreted as the area of the triangle; alternatively, as the distance of a pair of consecutive arrows to any potential candidate for their composite. These values may live in an arbitrary commutative quantale. Generalizing and extending recent work by Aliouche and Simpson, we give a condition for the existence of an Yoneda-type embedding which, in particular, gives the isometric embeddability of a metagory into a metrically enriched category. The generality of the value quantale allows for applications beyond the classical metric context.
\end{abstract}


\end{frontmatter}


\section{Introduction}
In Part III of \cite{Menger1928}, entitled {\em Entwurf einer Theorie der $n$-dimensionalen Metrik} (Sketch of a theory of the $n$-dimensional metric), Menger discusses various conditions to be satisfied by an $n$-{\em metric}. Chief among them is the {\em Simplexungleichung} (simplex inequality) and the {\em schwache Nullbedingung} (weak zero condition), which respectively generalize the triangle inequality and the condition that self-distances must be zero --- from the 1-dimensional to the $n$-dimensional case, that is: from a standard metric for pairs of points to functions of $(n+1)$-tuples of points, measuring the ``volume" of the body they span. These two conditions are also part of G\"ahler's \cite{Gahler1963} notion of {\em 2-metric space}, in which the simplex inequality becomes the {\em tetrahedral inequality}, and which gained renewed attention in the work of Aliouche and Simpson \cite{AlioucheSimpson2012}. In their paper \cite{AlioucheSimpson2017} on {\em approximate categorical structures}, these authors take the important step of allowing arbitrary arrows to replace, geometrically speaking,  the three straight edges of triangles spanned by a triple of points. So, they consider a graph (with directed edges and a distinguished loop at each vertex) that comes equipped with a real-valued function which measures the ``area" of all directed triangles
$$\bfig\Atriangle/<-`->`->/<500,300>[y`x`z;f`g`a]\efig$$
of fitting triples of (potentially rather ``wavy") arrows of the given graph, obeying, among other conditions, two tetrahedral-like inequalities and two self-distance axioms. The principal result of their paper gives sufficient conditions for the isometric embeddability of an approximate categorical structure $\mathbb X$ into a metrically enriched category in which, unlike in ${\mathbb X}$, fitting arrows may always be composed, and where the composition is compatible with the distance function for parallel arrows.

The category of (Fr\'echet) metric spaces and their, in a weak sense, {\em contractive} maps (also known as {\em non-increasing, short}, or {\em 1-Lipschitz} maps) is well known for its shortcomings, such as the absence of infinite products and of internal function spaces, both caused by the non-admittance of infinite distances. Even if one is interested only in ordinary metric spaces, it seems advantageous to work at least  initially in a larger, but categorically well-behaved, environment, a point amply emphasized by Lawvere \cite{Lawvere1973}. Doing otherwise leads, as in many of the works cited, to the addition of other auxiliary conditions that may also be problematic from a categorical point of view. Hence, in this work, when talking about real-valued metrics, we allow the value $\infty$ and impose just three conditions, namely that self-distances be zero, the triangle inequality hold, and distances be symmetric. This makes for a {\em topological} \cite{Adamek1990, Borceux1994, MonTop} and, hence, complete and cocomplete category over $\Set$ which, moreover is {\em symmetric monoidal closed} \cite{Kelly1982}, {\em i.e.}, allows for the formation of function spaces. Imposing the separation condition, thus requiring points with zero distance to be equal, destroys topologicity but still makes for a well-behaved reflective and, hence, complete and cocomplete  subcategory which is closed under the formation of function spaces. (The paper \cite{Weiss2012} offers further perspectives on generalized metrics.)

{\em Metric approximate categories}, as introduced in this paper and nick-named {\em metagories} (with stress on the second syllable) for short, follow these categorical guidelines and therefore  free the Aliouche-Simpson notion of approximate categorical structure of conditions that appear to be detrimental when studying the main properties of the structure. Hence, a $[0,\infty$]-valued metagory is a graph ({\em i.e.}, a ``category without composition"), equipped with an ``area'' function for directed triangles, which must satisfy the Aliouche-Simpson versions of Menger's (2-)simplex inequality and his weak zero condition. As already observed in the slightly more restrictive context of \cite{AlioucheSimpson2017}, every category enriched in the category of metric spaces (in the generalized categorical sense) gives a metagory, and the sets of parallel edges in a metagory carry a compatible metric (which, in particular, must automatically be symmetric). A major point that we add here to their theory is that (small) metagories, just like Lawvere metric spaces or small metrically enriched categories, form a symmetric monoidal-closed category, thus allowing for the formation of internal function spaces (see Theorem \ref{main thm}). To that end we offer a generalized notion of natural transformation between {\em contractors} (= morphisms) of metagories, which differs from a proposal in \cite{AlioucheSimpson2017} (see their discussion item 12.8). Armed with this tool, we are able to give the construction of the metagorical substitute for the Yoneda embedding of so-called {\em transitive} metagories, which in turn leads us to a generalization and alternative proof of the above-mentioned Aliouche-Simpson embeddability result (see Corollary \ref{main cor}). For that we utilize a well-known result by Wyler on lifting of adjoints along topological functors \cite{Wyler1971}.

The second major point of this paper is that, like in \cite{AkhvledianiClementinoTholen2010}, we allow ``distance" and ``area" values to live in an arbitrary commutative and unital {\em quantale}, not just in $[0,\infty]$. Hence, we replace the complete lattice $([0,\infty],\geq)$ and its monoid operation $+$ by any complete lattice $(\sV,\leq)$ which comes with a commutative and associative binary operation $\otimes$ preserving suprema in each variable, and which has a neutral element $\sk$. Hence, rather than being just non-negative real numbers or $\infty$, the values of the area function of a $\sV$-{\em metagory} could be logical truth values, elements of so-called frames (as intensively studied by Pultr and his coauthors, see in particular \cite{PicadoPultr2012}), or probability distribution functions (see \cite{MonTop}), to name only a few. This more general approach not only increases the potential range of applications of the theory developed, but equally importantly, it helps clarifying which properties of the rich structure of the (extended) real half-line are essential for establishing the fundamentals of a reasonably well-rounded theory. 

We conclude this article with a brief discussion of the above-mentioned transitivity condition vis-\`a-vis other potentially important subtypes of $\sV$-metagories, in particular those which allow one to find candidates for the composite of consecutive arrows up to a given error margin $\varepsilon$ (Section 8). We also offer (in Section 9) a potential blueprint for the notion of higher-dimensional approximate categorical structures, by pointing to the fact that the symmetric monoidal-cloased category of small $\sV$-metagories is enriched in itself, a full proof of which will appear in \cite{Tholen2019}.


\section{$\sV$-metric spaces and $\sV$-metric categories}
Throughout the paper $\sV$ denotes a {\em quantale} \cite{Rosenthal1990} which we tacitly assume to be commutative and unital; that is, $(\sV,\leq)$ is a complete lattice with a monoid structure, given by an associative and commutative binary operation $\otimes$ with a neutral element $\sk$, such that $\otimes$ distributes in each variable over arbitrary suprema:
$$u\otimes\bigvee_{i\in I}v_i=\bigvee_{i\in I}u\otimes v_i$$
for all $u, v_i \;(i\in I)$ in $\sV$. (Note that there is no requirement for $\sk$ vis-a-vis the least or largest elements in $\sV$.)
For every $u\in\sV$, since $u\otimes(-):\sV\to\sV$ preserves suprema, this map has a right adjoint, denoted by $u\multimap(-)$, whose value at $v\in\sV$ is characterized by
$$z\leq u\multimap v\iff z\otimes u\leq v$$
for all $z\in\sV$. A {\em homomorphism} of quantales is a map preserving all suprema and the monoid structure. {\bf Qnt} denotes the category of quantales and their homomorphisms.

\begin{exmp}\label{quantales}
(1) The {\em trivial} quantale $\sf 1$ has just one element and serves as a terminal object in {\bf Qnt}. The Boolean 2-chain $\sf 2=\{{\rm false}<{\rm true}\}$ with $\otimes={\rm min}$ is an initial object in {\bf Qnt}. The power set ${\sf 2}^X$ of a set $X$ becomes a quantale with $\otimes=\cap$ and serves as a direct product of $X$-many copies of $\sf 2$ in {\bf Qnt}. More generally, every {\em frame} ({\em i.e.}, every complete lattice in which finite infima distribute over arbitrary suprema) is a quantale. In particular, the open sets of a topological space form a quantale.
 
(2) The principal example of a quantale in this paper is the {\em Lawvere quantale} (see \cite{Lawvere1973}) ${\mathbb R}_+=([0,\infty],\geq, +,0)$, given by the extended real half line, ordered by $\geq$ (so that $\infty$ becomes the least and $0$ the largest element of the quantale), provided with the addition, naturally extended by 
$u+\infty=\infty +u=\infty$, for all $u\in[0,\infty]$. Then  $u\multimap v={\rm max}\{v-u,\,0\}$ for $u,v<\infty$, while $\infty\multimap v=0$ and $u\multimap \infty=\infty$ for all $v,u\in\sV, u<\infty.$ (These rules are easily memorized as the only homomorphism ${\sf 2}\to{\mathbb R}_+$ interprets $\infty$ as ${\rm false}$ and $\multimap$ as implication.)

(3) The quantale ${\mathbb R}_+$ is, of course, isomorphic to the unit interval $[0,1]$, ordered by the natural $\leq$ and provided with the multiplication.
Both, ${\mathbb R}_+$ and $[0,1]$ are embeddable into the quantale $ {\Delta}$ of all {\em distance distribution functions} $\varphi: [0,\infty]\to[0,1]$, required to satisfy the left-continuity condition $\phi(\beta)={\rm sup}_{\alpha<\beta}\phi(\alpha)$, for all $\beta\in [0,\infty]$. Its order is inherited from $[0,1]$, and its monoid structure is given by the commutative
{\em convolution} product $(\phi\otimes\psi)(\gamma)={\rm sup}_{\alpha+\beta\leq\gamma}\phi(\alpha)\psi(\beta)$;
the $\otimes$-neutral function $\kappa$ satisfies $\kappa(0)=0$ and $\kappa(\alpha)=1$ for all $\alpha >0$.

The quantale homomorphisms
$\sigma:[0,\infty]\to {{\Delta}}$
and $\tau:[0,1]\to {{\Delta}}$,
defined by $\sigma(\alpha)(\gamma)=0$ if $\gamma\leq\alpha$, and $1$ otherwise, and
$\tau(u)(\gamma)=u$ if $\gamma>0$, and $0$ otherwise, present ${\Delta}$ as a coproduct of
${\mathbb R}_+$ and $[0,1]$ in the category $\bf{Qnt}$, since every $\varphi\in {{\Delta}}$ has a presentation $\varphi={\rm sup}_{\gamma\in[0,\infty]}\sigma(\gamma)\otimes\tau(\varphi(\gamma))$.
\end{exmp}

The language built around $\sV$-metric spaces as described next is guided by the principal example (2), which requires us to {\em ``mentally invert the inequality sign of $\sV$" whenever we want to think of the given statement geometrically}. 

\begin{defn}\label{V-Met}
A $\sV$-{\em metric space} is a set $X$ together with a $\sV$-{\em metric}, {\em i.e.}, a map $d=d_X:X\times X\to \sV$ satisfying the conditions 
\begin{itemize}
\item[\;1.] $\sk\leq d(x,x),$
\item[\;2.] $d(x,y)\leq d(y,x),$
\item[\;3.] $d(x,y)\otimes d(y,z)\leq d(x,z),$
\end{itemize}
\noindent for all $x,y,z\in X$. (Of course, in condition 2 one actually has equality; same for condition 1, if $\sk$ is the top element of $\sV$.)
 $X$ is {\em separated} if, in addition, the implication

4. $\sk \leq d(x,y)\Lra x=y$

\noindent holds for all $x,y\in X$. A map $f:X\to Y$ of $\sV$-metric spaces is $(\sV$-){\em contractive}, or a {\em contraction}, if
$$d_X(x,x')\leq d_Y(fx,fx')$$
for all $x,x'\in X$; if $\leq$ may be replaced by $=$, $f$ is an {\em isometry}. The category of $\sV$-metric spaces and their contractions is denoted by ${\bf Met}_{\sV}$, and its full subcategory of separated spaces by ${\bf SepMet}_{\sV}$.
\end{defn}

It is easy to verify that ${\bf SepMet}_{\sV}$ {\em is reflective in} ${\bf Met}_{\sV}$: the {\em separated reflection of} $(X,d)$ provides the set $X/\!\sim$ obtained from the equivalence relation $(x\sim y\iff \sk\leq d(x,y))$, with the $\sV$-metric that makes the projection an isometry. 

\begin{exmp}
(1) ${\bf Met}_{\sf 1}$ is (equivalent to) the category $\Set$ of sets, while ${\bf Met}_{\sf 2}$ is the category of sets equipped with an equivalence relation; morphisms are maps preserving the equivalence relation of their domains. ${\bf SepMet}_{2}$ returns the category $\Set$ again.

(2) ${\bf Met}_{{\mathbb R}_+}$ is the category of {\em symmetric Lawvere metric spaces}, {\em i.e.}, of sets $X$ with a $[0,\infty]$-valued symmetric distance function satisfying the triangle inequality and making all self-distances $0$. Its separated objects differ from ordinary metric spaces only insofar as infinite distances are permitted.

(3) Under the isomorphism ${\mathbb R}_+\cong([0,1],\leq,\cdot, 1)$ of quantales, distances may be interpreted as  probabilities. Then ${\bf Met}_{\Delta}$ is the category of {\em probabilistic metric spaces}, studied under various degrees of generality in, for example, \cite{SchweizerSklar1983, HofmannReis2013, MonTop, Jager2015}.
\end{exmp}

\begin{rem}\label{V-metric remarks}
(1) The forgetful functor ${\bf Met}_{\sV}\to\Set$ is well known (see, for example, \cite{MonTop}) to be {\em topological} (in the sense of \cite{Adamek1990, Borceux1994}): for a family of set maps $f_i: X \to Y_i\;(i\in I)$, with every $Y_i$ carrying a $\sV$-metric $d_i$, one defines the {\em initial} $\sV$-metric $d$ on $X$ by
$$d(x,x')=\bw_{i\in I}d_i(f_ix,f_ix'),$$
for all $x, x'\in X$. Even when all $d_i$ are separated, $d$ need not be so. In order to make the $\sV$-metric space $X$ separated, one has to apply the separated reflector to $X$. As a topological category over $\Set$ (see Section 5), the category ${\bf Met}_{\sV}$ {\em is complete and cocomplete}, and so is its full reflective subcategory ${\bf SepMet}_{\sV}$.

(2) When one considers $\sV$ as a symmetric monoidal-closed category, with arrows given by $\leq$ and the ``internal hom" by $\multimap$, the objects of the category $\VCat$ of small $\sV$-{\em categories} ({\em i.e.}, categories enriched in $\sV$, see \cite{Lawvere1973, Kelly1982, MonTop}) may be described like the $\sV$-metric spaces, except that the symmetry condition 2 gets dropped; its morphisms, the $\sV$-{\em functors}, are defined like the $\sV$-contractions. This makes ${\bf Met}_{\sV}$ a full and coreflective subcategory of $\VCat$, with the coreflector symmetrizing the structure $a$ of a $\sV$-category $X$ to $d$, where $d(x,y) = a(x,y)\wedge a(y,x)$ for all $x,y\in X$. In particular, $\sV$ itself, which is a $\sV$-category with its internal hom $\multimap$, becomes a separated $\sV$-metric space, with its $\sV$-metric given by
$$d_{\sV}(u,v)=(u\multimap v)\wedge(v\multimap u).$$
Note that, for $\sV={\mathbb R}_+$ and $u,v<\infty$, $d_{\sV}(u,v)=|u-v|$ is the usual Euclidean distance.

(3) $\VCat$ is a symmetric monoidal-closed category, and ${\bf Met}_{\sV}$ is closed under forming tensor products and internal homs in $\VCat$. Hence, ${\bf Met}_{\sV}$ {\em is symmetric monoidal closed} again, with the monoidal and hom structures described, as follows:

$\bullet$ the tensor product $X\otimes Y$ has underlying set $X\times Y$, and its $\sV$-metric is given by $$d_{{\mathbb X}\otimes{\mathbb Y}}((x,y),(x',y'))=d_X(x,x')\otimes d_Y(y,y');$$ for $\sV={\mathbb R}_+$, this is the ``$+$-metric" on $X\times Y$;

$\bullet$ the hom-object $[X,Y]$ has underlying set ${\bf Met}_{\sV}(X,Y)$ ({\em i.e.} all $\sV$-contractions $X\to Y$), and its $\sV$-metric is given by $$d_{[X,Y]}(f,g)=\bw_{x\in X}d_Y(fx,gx)$$ (for $f,g:X\to Y$); for $\sV={\mathbb R}_+$, this is, of course, the ``sup-metric" on the function space of all contractions $X\to Y$.

The tensor-neutral space $\sf I$ is a singleton space whose only point has self-distance $\sk$ (the $\otimes$-neutral element of $\sV$). The tensor-hom adjunction means that, for all $\sV$-metric spaces $X,Y,Z$, a map $f:Z\otimes X\to Y$ is a $\sV$-contraction if, and only if, the map $f^{\sharp}:Z\to[X,Y]$ with $f^{\sharp}(z)(x)=f(z,x)$ is a well-defined $\sV$-contraction.

(4) For $\sV$-metric spaces $X$ and $Y$, separatedness gets inherited by $[X,Y]$ from $Y$. If $\sV$ is {\em integral} (so that the $\otimes$-neutral element $\sk$ is the top element in $\sV$) and $X$ and $Y$ are both separated, then also $X\otimes Y$ is separated.

(5) The $\sV$-metric $d$ of a $\sV$-metric space is contractive as a map $X\otimes X\to\sV$ where, of course, the domain is structured as in (3) and the codomain as in (2). This is an easy consequence of the symmetry and the triangle inequality for $d$.
\end{rem}

With the knowledge that ${\bf Met}_{\sV}$ is symmetric monoidal closed, it is natural to consider categories enriched in ${\bf Met}_{\sV}$.
\begin{defn}\label{metric cat}
A category $\mathbb X$ is called $\sV$-{\em metric} if it is enriched in ${\bf Met}_{\sV}$; that is, if for all $\mathbb X$-objects $x,y$, the hom-set ${\mathbb X}(x,y)$ is a $\sV$-metric space such that all composition maps
$${\mathbb X}(x,y)\otimes{\mathbb X}(y,z)\to{\mathbb X}(x,z)$$
are $\sV$-contractive; when denoting all metrics by $d$, this amounts to asking
$$d(f,f')\otimes d(g,g')\leq d(g\cdot f,g'\cdot f')$$
or, equivalently,
$$d(f,f')\leq d(g\cdot f,g\cdot f')\quad\text{and}\quad d(g,g')\leq d(g\cdot f,g'\cdot f)$$
for all $f,f':x\to y$ and $ g,g':y\to z$ in $\mathbb X$. The $\sV$-metric category $\mathbb X$ is {\em separated} if all of its hom-sets are separated. A functor $F:{\mathbb X}\to{\mathbb Y}$ of $\sV$-metric categories is $\sV$-{\em contractive} if $d(f,f')\leq d(Ff,Ff')$ for all $f,f':x\to y$ in $\mathbb X$, and {\em isometric} if $\leq$ may be replaced by $=$ in the last inequality. The category of all small $\sV$-metric categories and their $\sV$-contractive functors is denoted by ${\bf Met}_{\sV}$-$\Cat$. 
\end{defn}

\begin{exmp}
(1) Every $\sV$-metric space $X$ may be considered as a small $\sV$-metric category $\mathbb X$ with exactly two objects, $0,1$ say, and ${\mathbb X}(0,1)=X$; the only endomorphisms are identity morphisms whose self-distance is $\sk$. This construction extends functorially and embeds ${\bf Met}_{\sV}$ fully into ${\bf Met}_{\sV}$-$\Cat$.

(2) Being monoidal closed, ${\bf Met}_{\sV}$ is a $\sV$-metric category, with its hom-sets carrying the $\sV$-metric of Remark \ref{V-metric remarks}(3). Similarly, the category ${\bf Ban}_1$ of real (or complex) Banach spaces with linear operators of norm at most $1$ is an ${\mathbb R}_+$-metric category.
\end{exmp}

\begin{rem}\label{VMetCat}
For the same general reasons as given in Remark \ref{V-metric remarks}(3), the category ${\bf Met}_{\sV}$-$\Cat$ is symmetric monoidal closed:

$\bullet$ the tensor product ${\mathbb X}\otimes{\mathbb Y}$ of $\sV$-metric categories rests on the direct product ${\mathbb X}\times{\mathbb Y}$ of the underlying ordinary categories and provides its hom-sets with the $\sV$-metric given by
$$d_{{\mathbb X}\otimes{\mathbb Y}}((f,g),(f',g'))=d_{\mathbb X}(f,f')\otimes d_{\mathbb Y}(g,g'),$$
for all $f,f':x\to y$ in $\mathbb X$ and $g,g':z\to w$ in $\mathbb Y$;

$\bullet$ the hom-category $[{\mathbb X},{\mathbb Y}]$ of (small) $\sV$-metric categories has as underlying category the category of all $\sV$-contractive functors ${\mathbb X}\to{\mathbb Y}$ and their natural transformations; its hom-sets carry the $\sV$-metric defined by
$$d_{[{\mathbb X},{\mathbb Y}]}(\alpha,\alpha')=\bw_{x\in {\rm{ob}}{\mathbb X}}d_{\mathbb Y}(\alpha_x,\alpha'_x),$$
for all natural transformations $\alpha,\alpha':F\to G$ of $\sV$-contractive functors $F,G:{\mathbb X}\to {\mathbb Y}$; it is separated when $\mathbb Y$ is separated. 

The tensor-hom adjunction entails that, for all (small) $\sV$-metric categories ${\mathbb X},{\mathbb Y}, {\mathbb Z}$, the $\sV$-contractive functors
 $F:{\mathbb Z}\otimes{\mathbb X}\to{\mathbb Y}$ correspond naturally and bijectively to the $\sV$-contractive functors $F^{\sharp}:{\mathbb Z}\to[{\mathbb X},{\mathbb Y}]$.
\end{rem}

\section{$\sV$-metagories}
$\sV=(\sV,\leq,\otimes,\sk)$ continues to be a unital commutative quantale. To motivate the the key notion of this paper, we first note that a {\em G\"{a}hler-2-metric space} \cite{Gahler1963} may be described as a set $X$ that comes with an {\em area} (or {\em measure}) function $d:X\times X\times X\to{\mathbb R}$ satisfying the following conditions:
\begin{itemize}
\item[(a)] $\forall x,y\in X\; (\,x=y\iff\forall z\in X:d(x,y,z)=0\,);$
\item[(b)] $\forall x,y,z\in X:\, d(x,y,z)$ is independent of the order of the arguments $x,y,z$;
\item[(c)] $\forall x,y,z,w\in X: d(x,y,w)\leq d(x,y,z)+d(y,z,w)+d(x,z,w)$.
\end{itemize}
It is easy to show that $d$ takes necessarily non-negative values. A Lawvere-style generalization of this notion should allow $d$ to take the value $\infty$ and abandon the separation condition given by ``$\Leftarrow$" of (a), as well as the symmetry condition (b). However, as will become clear in \ref{metagory def} below, in the absence of (b) it becomes natural, and important, to state the two remaining conditions, namely that degenerate triangles have zero area (as in ``$\Rightarrow$" of (a)), and that the area of one face of a tetrahedron cannot exceed the sum of the areas of the other three faces (as in (c)), in two symmetric forms. Doing this in the $\sV$-context leads us to the following definition.
\begin{defn}\label{2Met def}
A $\sV$-{\em 2-metric space} is a set $X$ equipped with a function $d:X\times X\times X\to\sV$ satisfying the conditions
\begin{itemize}
\item[1.] $\sk\leq d(x,x,y),\quad\sk\leq d(x,y,y)$,
\item[2.] $d(x,y,z)\otimes d(y,z,w)\otimes d(x,y,w)\leq d(x,z,w),\quad d(x,y,z)\otimes d(y,z,w)\otimes d(x,z,w)\leq d(x,y,w),$
\end{itemize}
for all $x,y,z,w\in X$. A {\em contraction} $f:(X,d_X)\to (Y,d_Y)$ must satisfy
$$d_X(x,y,z)\leq d_Y(fx,fy,fz)$$
for all $x,y,z\in X$. This defines the category ${\bf 2Met}_{\sV}$.
\end{defn}

\begin{exmp}\label{2Met exmp}
(1) Every G\"{a}hler-2-metric space \cite{Gahler1963} 
is a ${\mathbb R}^+$-2-metric space. In particular, ${\mathbb R}^n$ is such a space when endowed with the Euclidean 2-metric, which assigns to each triple of points the area of the triangle they span.

(2) Here is a novel way of considering any $\sV$-metric space $(X,d)$ as ${\sf 2}$-2-metric space $(X,\tilde{d})$ (with the quantale $\sf 2$ of Example \ref{quantales}(1)), provided that $\sV$ is integral: for all $x,y,z\in X$, declare $\tilde{d}(x,y,z)$ to be true precisely when
$$d(x,y)\otimes d(y,z) = d(x,z)$$
(``\,$y$ lies on the line segment $\overline{xz}$\,"). The verification of the required conditions for $\tilde{d}$ is easy and, actually, works for any function $d:X\times X\to {\sV}$ satisfying $d(x,x)=\sk$ for all $x \in X$. Every isometry of $\sV$-metric spaces gives a contraction of the induced $\sf 2$-2-metric spaces.
\end{exmp}

We are now ready to give the $\sV$-modification of the Aliouche-Simpson notion (\cite{AlioucheSimpson2017}) of approximate categorical structure, which is subsumed by the general notion when we put $\sV={\mathbb R}_+$ and don't allow infinite values.
\begin{defn}\label{metagory def}
A $\sV$-$\underline{met}${\em ric} $\underline{a}${\em pproximate cate}$\underline{gory}$ or, for short, a $\sV$-{\em metagory}, $\mathbb X$, is given by

$\bullet$ a (potentially large) {\em graph}, consisting of a class ${\rm ob}{\mathbb X}$ of {\em vertices} or {\em objects} and a family of disjoint {\em hom-sets} ${\mathbb X}(x,y)$ of directed {\em edges} or {\em morphisms} $x\to y\;(x,y\in{\rm ob}{\mathbb X})$;

$\bullet$ a distinguished morphism $1_x$ in the hom set ${\mathbb X}(x,x)$ for every $x\in{\rm ob}{\mathbb X}$, called the {\em identity morphism} on $x$;

$\bullet$ a family $\delta=\delta_{\mathbb X}=(\delta_{x,y,z})_{x,y,z\in{\rm{ob}}{\mathbb X}}$ 
of {\em distance} or {\em area} functions
$$\delta=\delta_{x,y,z}:{\mathbb X}(x,y)\times{\mathbb X}(y,z)\times{\mathbb X}(x,z)\to\sV,$$
which are subject to two {\em identity laws} and two {\em associativity laws}:

\begin{tabular}{ll}
&\\
$\sk\leq\delta(1_x,f,f)$ (left identity law),&$\sk\leq\delta(f,1_y,f)$ (right identity law),\\
$\delta(f,g,a)\otimes\delta(g,h,b)\otimes\delta(f,b,c)\leq\delta(a,h,c)$&(left associativity law),\\
$\delta(f,g,a)\otimes\delta(g,h,b)\otimes\delta(a,h,c)\leq\delta(f,b,c)$&(right associativity law),\\
&
\end{tabular}

\noindent for all morphisms $f:x\to y,\,g:y\to z,\,h:z\to w, \,a:x\to z,\, b:y\to w,\,c:x\to w$.
\bigskip

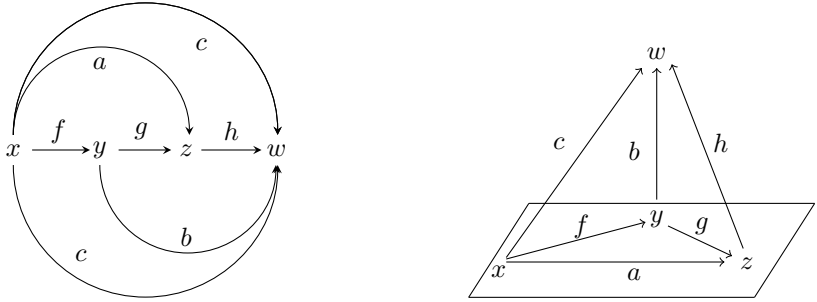
\begin{tabular}{lll}
\quad&\quad&\\
&$$\begin{tikzpicture}[>=stealth]
\path (0,0) node(x) {$x$}
(3.5,0) node(w) {$w$} 
(1.15,0) node(y) {$y$}
(2.3,0) node(z) {$z$}
(0.6,0.25) node(f) {$f$}
(1.7,0.25) node(g) {$g$}
(2.9,0.25) node(h) {$h$}
(1.15,1.15) node(a) {$a$}
(2.3,-1.15) node(b) {$b$}
(2.5,1.4) node(c) {$c$}
(0.9,-1.4) node(c') {$c$};

\draw [->] (0,.2) arc (180:0:50pt) ;
\draw [->] (0,-0.2) arc (-180:0:50pt);
\draw [->] (0,.2) arc (180:0:33.3pt) ;
\draw [->] (1.15,-0.2) arc (-180:0:33.3pt);
\draw [->] (0,.2) arc (180:0:50pt) ;
\draw[-> ] (0.25,0) -- (1.02,0);
\draw[-> ] (1.4,0) -- (2.1,0);
\draw[-> ] (2.5,0) -- (3.3,0);

\end{tikzpicture}$$
\quad\quad\quad\quad\quad\quad&
$$\begin{tikzpicture}
\draw (0,-0.25) - - (0.8,1);
\draw (0.8,1) - - (4.6,1);
\draw (0,-0.25) - - (3.8,-0.25);
\draw (3.8,-0.25) - - (4.6,1);

\path 
(0.4,0.1) node(x) {$x$}
(3.7,0.2) node(z) {$z$}
(2.5,0.8) node(y) {$y$}
(2.5,3) node(w) {$w$}
(1.5,0.7) node(f) {$f$}
(3.1,0.7) node(g) {$g$}
(2.2,0.06) node(a) {$a$}
(1.2,1.8) node(c) {$c$}
(2.2,1.7) node(b) {$b$}
(3.35,1.8) node(h) {$h$};

\draw[->] (2.65,0.7) -- (3.5,0.3);
\draw[->] (0.5,0.26) -- (2.35,0.75);
\draw[->] (0.5,0.22) -- (3.4,0.22);
\draw[->] (0.5,0.28) -- (2.3,2.8);
\draw[->] (2.5,1.05) -- (2.5,2.8);
\draw[->] (3.65,0.4) -- (2.7,2.85);

\end{tikzpicture}$$\\
&\quad&
\end{tabular}

\bigskip

\noindent We refer to the tetrahedron by the 6-tuple $(f,g,h;a,b,c)$.

A ($\sV$-){\em contractor} $F:{\mathbb X}\to{\mathbb Y}$ of $\sV$-metagories is a morphism of graphs, given by maps $F:{\rm ob}{\mathbb X}\to{\rm ob}{\mathbb Y}$ and $F=F_{x,y}:{\mathbb X}(x,y)\to{\mathbb Y}(Fx,Fy)\;(x,y\in{\rm ob}{\mathbb X})$, preserving identity morphisms and contracting areas ($x,f,g, a$ in ${\mathbb X}$ as above):
$$F(1x)=1_{Fx},\quad\quad\delta_{\mathbb X}(f,g,a)\leq\delta_{\mathbb Y}(Ff,Fg,Fa);$$
$F$ is {\em isometric} if $\leq$ can be replaced by $=$ in the last inequality. We denote by 
${\bf Metag}_{\sV}$ the category of small $\sV$-metagories ({\em i.e.}, those $\mathbb X$ with ${\rm ob}{\mathbb X}$ a set) and their $\sV$-contractors.
\end{defn}

\begin{exmp}\label{chaotic metag}
Every $\sV$-2-metric space $(X,d)$ defines the {\em chaotic} $\sV$-metagory $(\mathbb X,\delta)$ with ${\rm ob}{\mathbb X}=X$ and each hom-set ${\mathbb X}(x,y)$ containing exactly one morphism, which we call $(x,y)$; one simply puts 
$$\delta((x,y),(y,z),(x,z))=d(x,y,z)$$ for all $x,y,z\in X$. Obviously, as every contraction of $\sV$-2-metric spaces becomes a $\sV$-contractor of $\sV$-metagories, we obtain a full embedding
$${\bf 2Met}_{\sV}\to{\bf Metag}_{\sV},$$
which is actually reflective: its left adjoint provides the object set $X$ of a $\sV$-metagory $({\mathbb X},\delta)$ with the $\sV$-2-metric $d$ given by
$$d(x,y,z)=\bigvee_{f:x\rightarrow y, g:y\rightarrow z, a:x\rightarrow z}\delta(f,g,a)$$
for all $x,y.z\in X$.
\end{exmp}

In what follows, of greater importance than the full embedding of Example \ref{chaotic metag} is the faithful (but generally not full) {\em induced $\sV$-area functor}
$${\bf Met}_{\sV}\text{-}\Cat\to{\bf Metag}_{\sV}$$
whose definition we record next (generalizing \cite{AlioucheSimpson2017}, Prop. 5.7).
\begin{prop}\label{metag prop}
Every $\sV$-metric category $\mathbb X$ can be considered as a $\sV$-metagory when one puts $$\delta(f,g,a):=d(g\cdot f,a)$$ for all morphisms $f:x\to y,\,g:y\to z,\,a:x\to z$ in $\mathbb X$. The $\sV$-contractive functors of $\sV$-metric categories then become $\sV$-contractors.
\end{prop}
\begin{proof}
Trivially, $\sk\leq d(f,f)=d(f\cdot 1_x,f)=\delta(1_x,f,f)$, and likewise $\sk\leq\delta(f,1_y,f)$, for all $f:x\to y$. Using successively the symmetry, composition contractivity, and the triangle inequality of the $\sV$-metric $d$ one obtains the left associativity law for $\delta$ from
\begin{align*}
d(g\cdot f,a)\otimes d(h\cdot g,b)\otimes d(b\cdot f,c)&\leq d(a,g\cdot f)\otimes d(h\cdot g,b)\otimes d(b\cdot f,c)\\
& \leq d(h\cdot a,h\cdot g\cdot f)\otimes d(h\cdot g\cdot f, b\cdot f)\otimes d(b\cdot f,c)\\
& \leq d(h\cdot a, c),
\end{align*}
with all morphisms as in Definition \ref{metagory def}. Right associativity is shown similarly. The statement
on $\sV$-contractive functors is obvious.
\end{proof}

Next we point out that the hom-sets of a $\sV$-metagory actually carry a $\sV$-metric, considered in \cite{AlioucheSimpson2017} (Lemma 4.2) for $\sV={\mathbb R}_+$. For its presentation, it is conveneint to use the following terminology.

\begin{defn}
A $\sV$-{\em metric graph} $\mathbb X$ is a graph (as given by the first two bullet points of Definition \ref{metagory def}), such that each of its hom-sets ${\mathbb X}(x,y)$ carries a $\sV$-metric. Their morphisms are identity-preserving graph homorphisms whose hom-maps are $\sV$-contractive. ${\bf Met}_{\sV}\text{-}{\bf Gph}$ denotes the resulting category of small $\sV$-metric graphs.
\end{defn}

\begin{prop}\label{Vgraph prop}
The underlying graph of a $\sV$-metagory $\mathbb X$ becomes $\sV$-metric when one puts
$$d(f,f'):=\delta(1_x,f,f')=\delta(f,1_y,f')$$
for all $f,f':x\to y$ in $\mathbb X$. The $\sV$-contractors of $\sV$-metagories then become $\sV$-contractive morphisms of graphs.
\end{prop}

\begin{proof}
Let us first point out that the (implicitly claimed) identity $\delta(1_x,f,f')=\delta(f,1_y,f')$ follows from an application of the left and right associativity laws to the tetrahedron $(1_x,f,1_y;f,f,f')$. The identity laws give $d(f,f)\geq{\sk}$. Applying the left associativity law to $(f,1_y,1_y;f', 1_y,f)$ one obtains $d(f,f')\leq d(f',f)$. To prove the triangle inequality $d(f,f')\otimes d(f',f'')\leq d(f,f'')$ one may exploit the left associativity law of $(f',1_y,1_y;f,1_y,f'')$. The statement about $\sV$-contractors is obvious.
\end{proof}

\begin{cor} \label{metag cor}
Proposition ${\rm{\ref{Vgraph prop}}}$ describes a faithful functor $${\bf Metag}_{\sV}\to {\bf Met}_{\sV}\text{-}{\bf Gph}$$ which, when precomposed with the induced $\sV$-area functor ${\bf Met}_{\sV}\text{-}\Cat\to{\bf Metag}_{\sV}$, just produces the obvious forgetful functor ${\bf Met}_{\sV}\text{-}\Cat\to {\bf Met}_{\sV}\text{-}{\bf Gph}$.
\end{cor}

\begin{proof}
One only has to make sure that the $\sV$-metric of the hom-sets of a $\sV$-metric category does not get changed when passing first to the induced $\sV$-metagory and then to the induced $\sV$-metric graph. But this is obvious.
\end{proof}

With the hom-sets of the $\sV$-metagory ${\mathbb X}$ carrying the $\sV$-metric $d$ of Proposition \ref{Vgraph prop}, one can ask how a change in the hom-sets affects the area function $\delta$ of $\mathbb X$. The following answer generalizes Corollary 4.6 of \cite{AlioucheSimpson2017} to our context.

\begin{prop}\label{delta cont}
For all $f,f':x\to y,\,g,g':y\to z,\,a,a':x\to z$ in the $\sV$-metagory $\mathbb X$ one has
$$d(f,f')\otimes d(g,g')\otimes d(a,a')\otimes\delta(f,g,a)\leq \delta(f',g,',a').$$
\end{prop}

\begin{proof}
Applying the right associativity law to the tetrahedron $(f',1_y,g;f,g,a)$ one obtains
$$d(f,f')\otimes\delta(f,g,a)\leq\delta(f',1_y,f)\otimes\delta(1_y,g,g)\otimes\delta(f,g,a)\leq\delta(f',g,a).$$
Similarly one shows $d(g,g')\otimes\delta(f',g,a)\leq \delta(f',g',a)$ and $d(a,a')\otimes\delta(f',g',a)\leq\delta(f',g',a')$. The combination of all three inequalities gives the inequality claimed.
\end{proof}
In analogy to Remark \ref{V-metric remarks}(5), from Proposition \ref{delta cont} one concludes:
\begin{cor}
For all objects $x,y,z$ in a $\sV$-metagory $\mathbb X$, the area map $$\delta:{\mathbb X}(x,y)\otimes{\mathbb X}(y,z)\otimes{\mathbb X}(x,z)\to \sV$$ is $\sV$-contractive.
\end{cor}

One calls a $\sV$-metagory {\em separated} when its induced $\sV$-metric graph is separated, {\em i.e.}, if the $\sV$-metric of Proposition \ref{Vgraph prop} is separated. The functors of Corollary \ref{metag cor} may all be restricted to the respective full subcategories of separated objects, so that we have the following commutative diagram, in which the vertical arrows are inclusion functors::
$$\bfig
\square(700,0)/->`->`->`->/<900,400>[{\bf SepMet}_{\sV}\text{-}\Cat`{\bf SepMetag}_{\sV}`{\bf Met}_{\sV}\text{-}\Cat`
{\bf Metag}_{\sV};```]
\square(1600,0)/->`->`->`->/<900,400>[{\bf SepMetag}_{\sV}`{\bf SepMet}_{\sV}\text{-}{\bf Gph}`{\bf Metag}_{\sV}`{\bf Met}_{\sV}\text{-}{\bf Gph};```]
\efig$$
The following corollary shows that {\em the diagram stays commutative if the vertical arrows get replaced by their left adjoints}.

\begin{cor}
The separated reflection of a $\sV$-metagory is obtained by applying the separated reflection to each of its hom-sets and making the canonical projections isometries.
\end{cor}

\begin{proof}
The reflection ${\mathbb X}/\!\sim$ has the same objects as $\mathbb X$. Its hom-sets are the separated reflections of the hom-sets of the $\sV$-metagory $\mathbb X$, which are obtained with the equivalence relations described after Definition \ref{V-Met}, where $d$ denotes the $\sV$-metric induced by $\delta$ on each hom-set. Denoting by $[f]$ the equivalence class of $f\in{\mathbb X}(x,y)$, we derive from Proposition \ref{delta cont} that putting $\delta_{{\mathbb X}/\sim}([f],[g],[a]) = \delta(f,g,a)$ gives us a well-defined $\sV$-metagorical structure, which obviously has the required universal property.
\end{proof}

\section{Function spaces of $\sV$-metagories}
We want to show that ${\bf Metag}_{\sV}$ is, like ${\bf Met}_{\sV}\text{-}\Cat$, symmetric monoidal closed. It is easy to see that, for $\sV$-metagories ${\mathbb{X, Y}}$, we can define a tensor product ${\mathbb X}\otimes{\mathbb Y}$ analogously to the tensor product of $\sV$-metric categories, putting
${\rm{ob}}({\mathbb X}\otimes{\mathbb Y})={\rm ob}{\mathbb X}\times{\rm ob}{\mathbb Y}$ and
$${\mathbb X}\otimes{\mathbb Y}((x,y),(x',y'))={\mathbb X}(x,x')\times{\mathbb Y}(y,y'),\quad 1_{(x,y)}=(1_x,1_y),$$
$$\delta_{{\mathbb X}\otimes{\mathbb Y}}((f,h),(g,j),(a,b))=\delta_{\mathbb X}(f,g,a)\otimes\delta_{\mathbb Y}(h,j,b),$$
for all $f:x\to x', g:x'\to x'', a:x\to x''$ in $\mathbb X$ and $h:y\to y', j:y'\to y'', b:y\to y''$ in $\mathbb Y$. The $\otimes$-neutral $\sV$-metagory has exactly one object and one morphism, with $\sk$ as the only area value.

In order to be able to define a sutable hom-object $[{\mathbb X},{\mathbb Y}]$ we first need a notion of transformation of $\sV$-contractors.

\begin{defn}\label{nat trans def}
A {\em natural transformation} $\alpha:F\to G$ of $\sV$-contractors $F,\,G:{\mathbb X}\to{\mathbb Y}$ is given by a family $\alpha_f:Fx\to Gy$ of morphisms in $\mathbb Y$, where $f:x\to y$ runs through all morphisms of $\mathbb X$, such that, when we write $\alpha_x$ for $\alpha_{1_x}$,
$$\sk\leq \delta(Ff,\alpha_y,\alpha_f)\quad\text{and}\quad\sk\leq\delta(\alpha_x,Gf,\alpha_f).$$
One sees immediately that, when $\mathbb Y$ is induced by a separated $\sV$-metric category, then these two inequalities amount to the standard naturality condition $\alpha_y\cdot Ff = \alpha_f = Gf\cdot\alpha_x$.
$$\xymatrix{Fx
\ar[r]^{Ff}\ar[d]_{\alpha_x}\ar[rd]^{\alpha_f} & Fy
\ar[d]^{\alpha_y}\\Gx\ar[r]_{Gf} & Gy}$$
\end{defn}
For $\sV$-metagories $\mathbb{X, Y}$ we can now form the $\sV$-metagory{\footnote{As for ordinary categories, unless $\mathbb X$ is small, $[{\mathbb X},{\mathbb Y}]$ will generally live in a higher universe.}} $[{\mathbb X},{\mathbb Y}]$, putting 
$${\rm{ob}}[{\mathbb X},{\mathbb Y}]=\{F:{\mathbb X}\to{\mathbb Y}\,|\, F \;{\sV}\text{-contractor} \},\;
[{\mathbb X},{\mathbb Y}](F,G)=\{\alpha:F\to G\,|\,\alpha\text{ natural transformation}\,\},$$
$$1_F=(Ff)_{f:x\rightarrow y},\quad\delta_{[{\mathbb X},{\mathbb Y}]}(\alpha,\beta,\gamma)=\bw_{x\in{\rm{ob}}{\mathbb X}}\delta_{\mathbb Y}(\alpha_x,\beta_x,\gamma_x),$$
for all $\alpha:F\to G,\; \beta:G\to H, \;\gamma:F\to H$ in $[\mathbb X, \mathbb Y]$.

\begin{thm}\label{main thm}
$[\mathbb X,\mathbb Y]$ is a $\sV$-metagory, characterized by the property that for all $\sV$-metagories $\mathbb Z$, the $\sV$-contractors $F:{\mathbb Z}\otimes{\mathbb X}\to {\mathbb Y}$ correspond bijectively to the 
$\sV$-contractors $F^{\sharp}:{\mathbb Z}\to [\mathbb X,\mathbb Y]$, with the correspondence given by
$$(F^{\sharp}u)(x)=F(u,x),\;(F^{\sharp}u)(f)=F(1_u,f)\text{ and }(F^{\sharp}a)_f=F(a,f),$$
for all $f:x\to y$ in $\mathbb X$ and $a: u\to v$ in $\mathbb Z$.
\end{thm}

\begin{proof}
Trivially $\sk\leq\bw_{x\in{\rm{ob}}_{\mathbb X}}\delta(1_{Fx},\alpha_x,\alpha_x)=\delta(1_F,\alpha,\alpha)$ and $\sk\leq\delta(\alpha,1_G,\alpha)$, for all $\alpha:F\to G$. For every tetrahedron $(\lambda,\mu,\nu,\alpha,\beta,\gamma)$ in $[\mathbb X,\mathbb Y]$ and all $x\in{\rm{ob}}{\mathbb X}$ one has
$$\delta(\lambda,\mu,\alpha)\otimes\delta(\mu,\nu,\beta)\otimes\delta(\alpha,\nu,\gamma)\leq \delta(\lambda_x,\mu_x,\alpha_x)\otimes\delta(\mu_x,\nu_x,\beta_x)\otimes\delta(\alpha_x,\nu_x,\gamma_x)\leq\delta(\lambda_x,\beta_x,\gamma_x),$$
which implies the right associativity law; likewise for the left law.

Consider now a $\sV$-contractor $F:{\mathbb Z}\otimes{\mathbb X}\to {\mathbb Y}$. We must first make sure that $F^{\sharp}$ is well defined; that is: for $a:u\to v$ in $\mathbb Z$, $F^{\sharp}u: {\mathbb X}\to {\mathbb Y}$ is a $\sV$-contractor and $F^{\sharp}a:F^{\sharp}u\to F^{\sharp}v$ is a natural transformation. But these verifications are easy; for instance,
\begin{align*}
\delta((F^{\sharp}u)(f),(F^{\sharp}a)_y,(F^{\sharp}a)_f)&=\delta(F(1_u,f),F(a,1_y),F(a,f))\\
&\geq\delta((1_u,f),(a,1_y),(a,f))\\
&=\delta(1_u,a,a)\otimes\delta(f,1_x,f)\\
&\geq\sk\otimes\sk=\sk
\end{align*}
for all $f:x\to y$ in $\mathbb X$ gives half of the argument needed for the naturality of $F^{\sharp}u$. 

Next we show that $F^{\sharp}$ is a $\sV$-contractor. Since $(F^{\sharp}1_u)_f=F(1_u,f)=(F^{\sharp}u)(f)=(1_{F^{\sharp}u})_f$ for all objects 
$u$ in ${\mathbb Z}$ and morphisms $f$ in $\mathbb X$, $F^{\sharp}$ preserves identity morphisms. Furthermore, for all $a:u\to v,\, b:v\to w, \,c:u\to w$ in $\mathbb Z$ one has
\begin{align*}
\delta(F^{\sharp}a,F^{\sharp}b,F^{\sharp}c)&=\bw_{x\in{\rm{ob}}{\mathbb X}}\delta(F(a,1_x),F(b,1_x),F(c,1_x))\\
&\geq \bw_{x\in{\rm{ob}}{\mathbb X}}\delta((a,1_x),(b,1_x),(c,1_x))\\
&=\bw_{x\in{\rm{ob}}{\mathbb X}}\delta(a,b,c)\otimes \sk\;\geq\delta(a,b,c).
\end{align*}

Conversely, given the $\sV$-contractor $F^{\sharp}$, we must show the $\sV$-contractivity of $F$. For morphisms $a,b,c$ in $\mathbb Z$ as above and $f:x\to y,\,g:y\to z,\,h:x\to z$ in $\mathbb X$, the left associativity of the tetrahedron $((F^{\sharp}u)(f),(F^{\sharp}a)_y,(F^{\sharp}b)_g;(F^{\sharp}a)_f,(F^{\sharp}c)_g,(F^{\sharp}c)_h)$ and the naturality of $F^{\sharp}a$ show
\begin{align*}
&\;\quad\delta(F(a,f),F(b,g),F(c,h))=\delta((F^{\sharp}a)_f,(F^{\sharp}b)_g,(F^{\sharp}c)_h)\\
&\geq\delta((F^{\sharp}u)(f),(F^{\sharp}a)_y,(F^{\sharp}a)_f)\otimes\delta((F^{\sharp}a)_y,(F^{\sharp}b)_g,(F^{\sharp}c)_g)\otimes\delta((F^{\sharp}u)(f),(F^{\sharp}c)_g,(F^{\sharp}c)_h)\\
&\geq\sk\otimes \delta((F^{\sharp}a)_y,(F^{\sharp}b)_g,(F^{\sharp}c)_g)\otimes\delta((F^{\sharp}u)(f),(F^{\sharp}c)_g,(F^{\sharp}c)_h).
\end{align*}
Now, the right associativity of the tetrahedra $((F^{\sharp}a)_y, (F^{\sharp}b)_y, (F^{\sharp}w)(g); (F^{\sharp}c)_y,(F^{\sharp}b)_g,(F^{\sharp}c)_g)$ and 
$((F^{\sharp}u)(f),(F^{\sharp}u)(g),(F^{\sharp}c)_z; (F^{\sharp}u)(h),(F^{\sharp}c)_g,(F^{\sharp}c)_h)$, together with multiple naturality applications, shows first
\begin{align*}
&\;\quad\delta((F^{\sharp}a)_y,(F^{\sharp}b)_g,(F^{\sharp}c)_g)\\
&\geq\delta((F^{\sharp}a)_y,(F^{\sharp}b)_y,(F^{\sharp}c)_y)\otimes\delta((F^{\sharp}b)_y,(F^{\sharp}w)(g),(F^{\sharp}b)_g)\otimes\delta((F^{\sharp}c)_y,(F^{\sharp}w)(g),(F^{\sharp}c)_g)\\
&\geq\delta(a,b,c)\otimes\sk\otimes\sk
\end{align*}
and then, since $F^{\sharp}u$ is $\sV$-contractive,
\begin{align*}
&\;\quad\delta((F^{\sharp}u)(f),(F^{\sharp}c)_g,(F^{\sharp}c)_h)\\
&\geq\delta((F^{\sharp}u)(f), (F^{\sharp}u)(g),(F^{\sharp}u)(h))\otimes\delta((F^{\sharp}u)(g),(F^{\sharp}c)_z,(F^{\sharp}c)_g)\otimes\delta((F^{\sharp}u)(h),(F^{\sharp} c)_z,(F^{\sharp}c)_h)\\
&\geq\delta(f,g,h)\otimes\sk\otimes\sk.
\end{align*}
Consequently,
$\delta(F(a,f),F(b,g),F(c,h))\geq\delta(a,b,c)\otimes\delta(f,g,h)=\delta((a,f),(b,g),(c,h)),$
as desired.
\end{proof}

\begin{cor} 
The category ${\bf Metag}_{\sV}$ is symmetric monoidal closed. The induced $\sV$-area functor\\ ${\bf Met}_{\sV}\text{-}\Cat\to{\bf Metag}_{\sV}$ preserves tensor products and the $\otimes$-neutral object, but it generally fails to preserve the internal  hom-objects.
\end{cor}

\begin{rem}\label{sep remark}
{\em For a separated $\sV$-metric category $\mathbb Y$ considered as a $\sV$-metagory, and for any $\sV$-metagory $\mathbb X$, also the $\sV$-metagory $[\mathbb X, \mathbb Y]$ is induced by a separated $\sV$-metric category.} Indeed, as already remarked in Definition \ref{nat trans def}, under the hypothesis on $\mathbb Y$ every natural transformation $\alpha: F\to G:{\mathbb X}\to{\mathbb Y}$ satisfies the naturality condition $\alpha_y\cdot Ff = \alpha_f = Gf\cdot\alpha_x$ for all $f:x\to y$ in $\mathbb X$ and is therefore completely determined by the morphisms $\alpha_x: Fx\to Gx \;(x\in{\rm{ob}}{\mathbb X})$. Furthermore, one may then define a composition for natural transformations of contractors as one does for natural transformations of functors, and finally define a $\sV$-metric as one does in 
${\bf Met}_{\sV}\text{-}\Cat$ (see Remark \ref{VMetCat}). This makes $[\mathbb X,\mathbb Y]$ a $\sV$-metric category, and one easily verifies that its $\sV$-metric induces the $\sV$-metagorical structure described by Theorem \ref{main thm}.
\end{rem}

\section{Topologicity and free structures}
In Corollary \ref{metag cor} we considered the upper horizontal row of functors of the following commutative and self-explanatory diagram of categories and forgetful functors,

$$\bfig
\square(800,0)/->`->`->`->/<700,400>[{\bf Met}_{\sV}\text{-}\Cat`{\bf Metag}_{\sV}`\Cat`{\bf Gph};```]
\square(1500,0)/->`->`->`->/<700,400>[{\bf Metag}_{\sV}`{\bf Met}_{\sV}\text{-}{\bf Gph}`{\bf Gph}`{\bf Gph};```{\rm{Id}}]
\efig$$
for which we would like to construct left-adjoints. To this end, but also in order to better understand the structure of the categories at issue, it is useful to observe first that all three vertical functors are {\em topological} \cite{Adamek1990, Borceux1994, MonTop}, with initial structures being constructed as for ${\bf Met}_{\sV}\to {\bf Set}$ (see \ref{V-metric remarks} (1)). Indeed, recall that for a set $X$ and any (possibly large) family of maps $f_i:X\to Y_i$ to $\sV$-metric spaces $(Y_i,d_i)\;(i\in I)$, the initial structure on $X$ is given by
$$d(x,x')=\bw_{i\in I}d_i(f_ix,f_ix').$$
Now, replacing the set $X$ by a small\footnote{Being small is not essential for the existence of initial structures.}  category $\mathbb X$, the $\sV$-metric spaces $Y_i$ by $\sV$-metric categories $(\mathbb Y_i,d_i)$, and the maps $f_i$ by functors $F_i:{\mathbb X}\to{\mathbb Y}_i$, the  formula 
$$d(f,f')=\bw_{i\in I}d_i(F_if,F_if')$$
describes the initial structure on $\mathbb X$ with respect to the left vertical functor of the diagram above; likewise for the right vertical functor, and analogously for the middle vertical functor:
$$\delta(f,g,a)=\bw_{i\in I}\delta_i(F_if,F_if',F_ia)$$
describes the initial $\sV$-metagorical structure on the graph $\mathbb X$ for a family of graph morphisms $F_i$ to $\sV$-metagories $(\mathbb Y_i,\delta_i)\;(i\in I)$.

Now, let us recall Wyler's {\em Taut Lift Theorem} (see \cite{Wyler1971}, \cite{Tholen1978}, or Section II.5.11 of \cite{MonTop}) and consider any commutative diagram
$$\bfig
\square(800,0)<600,300>[{\mathcal A}`{\mathcal B}`{\mathcal X}`{\mathcal Y};G`U `V`J]
 \efig$$
 of categories and functors. Wyler proved the ``if"-part of the following statement; for ``only if" and a generalization, see \cite{Tholen1978}.
 \begin{prop}\label{Wyler}
If $U$ is topological and $J$ has a left adjoint $H$, then $G$ has a left adjoint $F$ with $UF=HV$ if, and only if, $G$ preserves initiality, in the sense that it maps $U$-initial families to $V$-initial families. In this case, $V$ maps $G$-universal arrows to $J$-universal arrows.
\end{prop}
 
The given formulae for initial structures show that the the two upper horizontal functors in our first diagram preserve initiality, making Proposition \ref{Wyler} applicable to them. Hence, we obtain:

\begin{prop}\label{topological}
The categories ${\bf Met}_{\sV}\text{-}{\bf Cat}$, ${\bf Metag}_{\sV}$ and ${\bf Met}_{\sV}\text{-}{\bf Gph}$ are topological over $\Cat, {\bf Gph}$ and ${\bf Gph}$, respectively, in particular complete and cocomplete. The left adjoints to $\Cat\to{\bf Gph}$ and ${\rm{Id}}_{\bf Gph}$ can be lifted along the topological functors to left adjoints of the functors
$${\bf Met}_{\sV}\text{-}\Cat\to{\bf Metag}_{\sV}\to {\bf Met}_{\sV}\text{-}{\bf Gph}.$$
\end{prop}

\begin{rem}
(1) According to Wyler's construction, the underlying category of the free $\sV$-metric category ${\sf Path}{\mathbb X}$ over a given $\sV$-metagory $\mathbb X$ is the free category over the underlying graph of $\mathbb X$ (consisting of all paths in the graph, such that the given identity morphisms get identified with empty paths); its $\sV$-metric makes the insertion ${\mathbb X}\to {\sf Path}{\mathbb X}$ $\sV$-contractive, which then serves as the unit of the adjunction at $\mathbb X$. Our goal is to show that the insertion is actually an isometry, by taking advantage of a generalized Yoneda embedding of $\mathbb X$, which we present next.

(2) Aliouche and Simpson \cite{AlioucheSimpson2017} gave an {\em ad-hoc} description of the metric of ${\sf Path}{\mathbb X}$ in case $\sV={\mathbb R}_+$. To confirm that ${\mathbb X}\to{\sf Path}{\mathbb X}$ becomes an isometry, they resort to the effect of the Yoneda construction (see Section 7) on objects, without establishing the construction fully.
\end{rem}

\section{$\sV$-distributors}
Let us first shed some further light on the separated $\sV$-metric space $\sV$ with its $\sV$-metric $d$, defined 
by $d_{\sV}(u,v)=(u\multimap v)\wedge(v\multimap u)$ as in Remark \ref{V-metric remarks}(2), and make explicit its compatibility with the $\otimes$-operation. For $\sV={\mathbb R}_+$, this compatibility condition amounts precisely to the fact that the addition $[0,\infty]\times[0,\infty]\to[0,\infty]$ is contractive when, based on the Euclidean distance of the extended real half-line, the domain carries the $+$-metric.

\begin{prop}\label{V monoid}
The binary operation $\otimes$ of $\sV$ is a $\sV$-contractive map $\sV\otimes\sV\to\sV$, that is:
$$d_{\sV}(u,v)\otimes d_{\sV}(w,z)\leq d_{\sV}(u\otimes w,v\otimes z)$$
for all $u,v,w,z\in\sV$.
Consequently, $(\sV,\otimes,\sk)$ is an internal monoid in the monoidal category ${\bf Met}_{\sV}$.
\end{prop}

\begin{proof}
Since $$(u\multimap v)\otimes(w\multimap z)\otimes u\otimes w= ((u\multimap v)\otimes u)\otimes((w\multimap z)\otimes w)\leq v\otimes z,$$ one obtains
$(u\multimap v)\otimes(w\multimap z)\leq (u\otimes w)\multimap (v\otimes z)$ and, likewise, $(v\multimap u)\otimes(z\multimap w)\leq(v\otimes z)\multimap(u\otimes w)$. Consequently,
\begin{align*}
d(u,v)\otimes d(w,z) &=  ((u\multimap v)\otimes(v\multimap u))\wedge ((w\multimap z)\otimes(z\multimap w))\\
&\leq ((u\multimap v)\otimes(w\multimap z))\wedge  ((v\multimap u)\otimes(z\multimap w))\\
&\leq ((u\otimes w)\multimap (v\otimes z))\wedge ((v\otimes z)\multimap(u\otimes w))\\
& = d(u\otimes w,v\otimes z).
\end{align*}
\end{proof}

For $\sV$-metric spaces $X,Y$, a $\sV$-{\em distributor} \cite{Borceux1994} (also: $\sV$-{\em bimodule} \cite{Lawvere1973}) $\varphi:X\rto Y$ is given by a family $(\phi(x,y))_{x\in X,y\in Y}$ of values in $\sV$ satisfying the conditions
$$\phi(x,y)\otimes d_Y(y,y')\leq\phi(x,y'),\quad d_X(x',x)\otimes\phi(x,y)\leq\phi(x',y),\quad(*)$$
for all $x,x'\in X,\; y,y'\in Y$; the composition rule for $\phi$ followed by $\psi:Y\rto Z$ says
$$(\psi\circ\phi)(x,z)=\bv_{y\in Y}\phi(x,y)\otimes\psi(y,z).$$
Since $(*)$ gives $d_Y\circ\phi=\phi$ and $\phi\circ d_X=\phi$, the structure of a $\sV$-metric space $X$ serves as the identity morphism on $X$ in the category ${\bf Dist}_{\sV}$ of $\sV$-metric spaces and their $\sV$-distributors. Every $\sV$-contractive map $f:X\to Y$ induces the $\sV$-distributors $f_*:X\rto Y,\; f^*:Y\rto X$, given by
$$f_*(x,y)=d_Y(fx,y),\;f^*(y,x)=d_Y(y,fx)$$
for all $x\in X,\,y\in Y$. The assignments $f\mapsto f_*,\;f\mapsto f^*$ define identity-on-objects functors
$${\bf Met}_{\sV}\to {\bf Dist}_{\sV}\longleftarrow({\bf Met}_{\sV})^{\op}.$$

Every $v\in \sV$ may be considered a $\sV$-distributor $v:{\sf I}\rto{\sf I}$ of the one-point $\otimes$-neutral $\sV$-metric space $\sf I$ which, when we consider $\sV$ with its monoid structure as a one-object category, defines a full embedding
$$\sV\to{\bf Dist}_{\sV}.$$
This embedding becomes an isometry when we extend the $\sV$-metric $d= d_{\sV}$ of $\sV$ to the hom-sets of ${\bf Dist}_{\sV}$, by putting
$$d(\phi,\phi'):=\bw_{x\in X,y\in Y}d(\phi(x,y),\phi'(x,y)),$$
for all $\phi,\phi':X\rto Y,\, X,Y$ in ${\bf Met}_{\sV}$. One easily sees that ${\bf Met}_{\sV}(X,Y)$ has now become a $\sV$-metric space. More importantly, in generalization of Proposition \ref{V monoid}, the composition map
$${\bf Dist}_{\sV}(X,Y)\otimes{\bf Dist}_{\sV}(Y,Z)\to{\bf Dist}_{\sV}(X,Z)$$
is $\sV$-contractive, for all $\sV$-metric spaces $X,Y,Z$, as we show next.

\begin{prop}\label{dist}
${\bf Dist}_{\sV}$ is a ${\bf Met}_{\sV}$-enriched category; its hom-sets are separated $\sV$-metric spaces.
\end{prop}

\begin{proof}
In order to confirm the $\sV$-contractivity of the composition map, we must show
$$d(\phi,\phi')\otimes d(\psi,\psi')\leq d(\psi\circ\phi,\psi'\circ\phi'),$$
for all $\phi,\phi':X\rto Y,\, \psi,\psi':Y\rto Z$. By definition of $d(\phi,\phi')$ one has $d(\phi,\phi')\leq \phi(x,y)\multimap \phi'(x,y)$
 and, hence, $d(\phi,\phi')\otimes\phi(x,y)\leq \phi'(x,y)$, for all $x\in X,\, y\in Y$. Likewise, $d(\psi,\psi')\otimes\psi(y,z)\leq\psi'(y,z),$ for all $y\in Y,\,z\in Z$. Consequently, for all $x\in X,\, z\in Z$ one obtains
 $$\bv_{y\in Y}d(\phi,\phi')\otimes d(\psi,\psi')\otimes\phi(x,y)\otimes\psi(y,z)\leq\bv_{y \in Y}\phi'(x,y)\otimes\psi'(x,y).$$ Since $\otimes$ distributes over joins, this means
$d(\phi,\phi')\otimes d(\psi,\psi')\otimes (\psi\circ\phi)(x,z)\leq(\psi'\circ\phi')(x,z)$, or
$$d(\phi,\phi')\otimes d(\psi,\psi")\leq(\psi\circ\phi)(x,z)\multimap(\psi'\circ\phi')(x,z)$$
for all $x\in X, z\in Z$. For symmetry reasons, this gives the inequality
$$d(\phi,\phi')\otimes d(\psi,\psi')\leq\bw_{x\in X, z\in Z}((\psi\circ\phi)(x,z)\multimap(\psi'\circ\phi')(x,z))\wedge((\psi'\circ\phi')(x,z)\multimap(\psi\circ\phi)(x,z)),$$
as desired. The hom-sets ${\bf Dist}_{\sV}(X,Y)$ inherit separatedness from $\sV$.
\end{proof}

\begin{rem}\label{opposite}
For every $\sV$-metagory $\mathbb X$ one defines its {\em opposite $\sV$-metagory} ${\mathbb X}^{\rm{op}}$ by
$${\rm{ob}}({\mathbb X}^{\op})={\rm{ob}}{\mathbb X},\quad{\mathbb X}^{\op}(y,x)={\mathbb X}(x,y),\quad\delta_{{\mathbb X}^{\op}}(g,f,a)=\delta_{\mathbb X}(f,g,a),$$
for all $f:x\to y,\,g:y\to z,\,a:x\to z$ in $\mathbb X$.
\end{rem}
In conjunction with Remark \ref{sep remark} we conclude:
\begin{cor}
For every (small) $\sV$-metagory $\mathbb X$, the $\sV$-metagory $[{\mathbb X}^{\op},{\bf Dist}_{\sV}]$ is (induced by) a separated $\sV$-metric category.
\end{cor}

\section{The Yoneda $\sV$-contractor}
For a (small) $\sV$-metagory $\mathbb X$, we would like to establish an injective isometry
$$\sy:{\mathbb X}\to[{\mathbb X}^{\op},{\bf Dist}_{\sV}]$$
in such a way that, when $\mathbb X$ is actually a $\sV$-metric category, $\sy$ factors through the Yoneda embedding of $\mathbb X$ into $[{\mathbb X}^{\op},{\bf Met}_{\sV}]$. To this end we will have to require an additional condition on the metagory $\mathbb X$, termed {\em absolute (left/right) transitivity} in \cite{AlioucheSimpson2017} in case $\sV={\mathbb R}_+$; here we leave off the adjective absolute:
\begin{defn}\label{trans def}
A $\sV$-metagory $\mathbb X$ is {\em left transitive} if 
$$\delta(f,b,c)\otimes\delta(g,h,b)\leq\bv_{a:x\rightarrow z}\delta(f,g,a)\otimes\delta(a,h,c),$$
for all $f:x\to y, g:y\to z, h:z\to w, b:y\to w, c:x\to w$ in $\mathbb X$, and {\em right transitive} if
$$\delta(f,g,a)\otimes\delta(a,h,c)\leq\bv_{b:y\rightarrow w}\delta(f,b,c)\otimes\delta(g,h,b)$$
for all $f,g,h,c$ as above and $a:x\rightarrow z$ in $\mathbb X$. {\em Transitivity} is the conjunction of both properties and amounts to
$$\bv_{a:x\rightarrow z}\delta(f,g,a)\otimes\delta(a,h,c)=\bv_{b:y\rightarrow w}\delta(f,b,c)\otimes\delta(g,h,b),$$ for all $f,g,h,c$ as above.
\end{defn}

\begin{rem}\label{trans remarks}
(1) {\em Every $\sV$-metagory induced by a $\sV$-metric category is transitive.} To see that it is right transitive, given $f,g,h,a,c$ as above, just consider $b:=h\cdot g$; likewise for left transitivity.

(2) A $\sV$-metagory $\mathbb X$ which comes with an injective isometry into a transitive $\sV$-metagory $\mathbb Y$ may fail to be right (or left) transitive, even when $\mathbb Y$ is induced by a separated $\sV$-metric category (unless $\sV$ is the trivial quantale $\sf 1$). Just consider the 4-object $\sV$-metagory $\mathbb X$ whose only non-identical arrows are $f:x\to y, g:y\to z,h:z\to w, a:x\to z, c:x\to w$, and let 
$\mathbb Y$ have just one additional arrow, $b:y\to w$; then, letting all areas/distances equal $\sk$, the inclusion ${\mathbb X}\hookrightarrow{\mathbb Y}$ confirms the claim.

(3) There is an ${\mathbb R}_+$-metagory for which there is no injective isometry mapping it into an ${\mathbb R}_+$-metric category: see Example 6.6 of \cite{AlioucheSimpson2017}. Hence, our goal is not reachable without additional conditions on the given metagory.
\end{rem}

For any $\sV$-metagory $\mathbb X$ and every object $w$ and morphism $f:x\to y$ in ${\mathbb X}$, one has the $\sV$-distributor
$$\sy_w(f):{\mathbb X}(y,w)\rto{\mathbb X}(x,w),\quad
\sy_w(f)(b,c)=\delta(f,b,c),$$ for all $b:y\to w, c:x\to w$. Indeed, given another morphism $c':x\to w$, since $\sk\leq\delta(1_x,f,f)$, left associativity applied to the tetrahedron $(1_x,f,b;f,c,c')$ shows
$$\sy_w(f)(b,c)\otimes d(c,c')\leq\delta(1_x,f,f)\otimes\delta(f,b,c)\otimes \delta(1_x,c,c')\leq\sy_w(f)(b,c').$$
This is one of the two inequalities to be checked; the other follows similarly.

\begin{prop}
For every object $w$ in a left-transitive $\sV$-metagory $\mathbb X$ one has the $\sV$-contractor
$$\sy_w:{\mathbb X}^{\rm{op}}\to{\bf Dist}_{\sV},\quad x\mapsto {\mathbb X}(x,w).$$
\end{prop}

\begin{proof}
With $d$ denoting the $\sV$-metric of the hom-sets of $\mathbb X$, first we observe
$$\sy_w(1_x)=\delta(1_x,c,c')=d(c,c')$$
for all $c,c':x\to w$, which shows that $\sy_w$ preserves identity morphisms. Next,
we need to confirm the inequality $\delta(f,g,a)\leq d(\sy_w(f)\circ\sy_w(g),\sy_w(a))$, for all $f:x\to y, g:y\to z, a:x\to z$ in $\mathbb X$. To this end, for all morphisms $h:z\to w, c:x\to w$, on one hand the left associativity law gives
$$\delta(f,g,a)\otimes\bv_{b:y\rightarrow w}(\delta(g,h,b)\otimes\delta(f,b,c))\leq\delta(a,h,c),$$
which implies 
$\delta(f,g,a)\leq(\sy_w(f)\circ\sy_w(g))(h,c)\multimap\delta(a,h,c)).$ On the other hand, from the right transitivity hypothesis one has $\delta(f,g,a)\leq \delta(a,h,c)\multimap (\sy_w(f)\circ\sy_w(g))(h,c)$. Combining the last two inequlaities we obtain
$$\delta(f,g,a)\leq \bw_{h:z\rightarrow w,c:x\rightarrow w}((\sy_w(f)\circ\sy_w(g))(h,c)\multimap\sy_w(a)(h,c))\wedge ( \sy_w(a)(h,c)\multimap (\sy_w(f)\circ\sy_w(g))(h,c)),$$
which is in fact the inequality that needed to be confirmed.
\end{proof}

\begin{prop}
For every morphism $m:w\to v$ in a transitive $\sV$-metagory $\mathbb X$ one has a natural transformation
$\sy_m:\sy_w\to\sy_v$ with
$$(\sy_m)_x:{\mathbb X}(x,w)\rto{\mathbb X}(x,v),\quad(\sy_m)_x(c,e)=\delta(c,m,e)$$
for all $c:x\to w,\,e:x\to v$ in $\mathbb X$. For $m=1_w$, $\sy_m$ is the identity transformation of $\sy_w$.
\end{prop}

\begin{proof}
It suffices to show that, for every morphism $f:x\to y$, the naturality diagram
$$\bfig
\square(800,0)<700,300>[{\mathbb X}(y,w)`{\mathbb X}(x,w)`{\mathbb X}(y,v)`{\mathbb X}(x,v);(\sy_w)(f)`(\sy_m)_y`(\sy_m)_x`(\sy_v)(f)]
\place(800,150)[\text{-}]
\place(1500,150)[\text{-}]
\place(1150,0)[\shortmid]
\place(1150,300)[\shortmid]
 \efig$$
 commutes in $\bf Dist_{\sV}$,  since then one can set $(\sy_m)_f$ as the common composite given by the NW-SE diagonal of the diagram. But the commutativity of every such diagram amounts precisely to stating that $\mathbb X$ is transitive. For $m=1_w$ one has
 $$(\sy_{1_w})_x(c,e)=\delta(c,1_w,e)=d(c,e)=\delta(1_x,c,e)=\sy_w(1_x)(c,e)=(1_{\sy_w})_x(c,e).$$
\end{proof}

\begin{thm}\label{Yoneda contractor}
Let $\mathbb X$ be a transitive $\sV$-metagory. Then the {\em Yoneda $\sV$-contractor}
$$\sy:{\mathbb X}\to[{\mathbb X}^{\op},{\bf Dist}_{\sV}]$$
of $\mathbb X$ is an isometry, mapping $\mathbb X$ into a $\sV$-metric category. It maps objects injectively, and the same is true for morphisms if $\mathbb X$ is separated.
\end{thm}
\begin{proof}
Considering morphisms $c:x\to w,\,m:w\to v,\,n:v\to u, \,p:w\to u,\,j:x\to u$ in $\mathbb X$, from the right associativity law and the transformation of joins into meets by the maps $(\text{-})\multimap t:\sV\to \sV$ for every $t\in \sV$ we obtain
\begin{align*}
\delta(m,n,p)&\leq\bw_{e:x\rightarrow v}((\delta(e,n,j)\otimes\delta(c,m,e))\multimap\delta(c,p,j))\\
&\leq(\bv_{e:x\rightarrow v}(\sy_n)_x(e,j)\otimes(\sy_m)_x(c,e))\multimap(\sy_p)_x(c,j)\\&=((\sy_n)_x\circ(\sy_m)_x)(c,j)\multimap(\sy_p)_x(c,j),
\end{align*}
while the left transitivity law gives
$$\delta(m,n,p)\leq(\sy_p)_x(c,j)\multimap ((\sy_n)_x\circ(\sy_m)_x)(c,j).$$
Consequently,
$$\delta(m,n,p)\leq\bw_{x\in{\rm{ob}}{\mathbb X}}\bw_{c:x\rightarrow w,j:x\rightarrow u}d(((\sy_n)_x\circ(\sy_m)_x)(c,j),(\sy_p)_x(c,j))=d(\sy_n\cdot\sy_m,\sy_p),$$
which shows the $\sV$-contractivity of $\sy$. For the reverse inequality we note that, when considering $w=x, c=1_x, j=p$, since $\delta(1_x,p,p)\leq\sk$ and $\sk\multimap t=t$ for all $t\in\sV$, one obtains
\begin{align*}
d(\sy_n\circ\sy_m,\sy_p)&\leq d(((\sy_n)_x\circ(\sy_m)_x)(1_x,p),(\sy_p)_x(1_x,p))\\
&\leq\delta(1_x,p,p)\multimap(\bv_{e:x\rightarrow v}\delta(1_x,m,e)\otimes\delta(e,n,p))\\
&\leq\bv_{e:x\rightarrow v}\delta(1_x,e,m)\otimes\delta(e,n,p)\otimes\delta(1_x,p,p)\leq\delta(m,n,p);
\end{align*}
here the last inequality follows from the left associativity law applied, for every $e$, to the tetrahedron $(1_x,e,n;m,p,p)$.

The argument for $\sy$ mapping objects injectively is standard: from $\sy_w=\sy_v$ one obtains $1_w\in{\mathbb X}(w,w)=\sy_w(w)=\sy_v(w)={\mathbb X}(w,v)$, which is possible only if $w=v$. For $m,n:w\to v$, suppose now that $\sy_m=\sy_n$; then $(\sy_m)_w=(\sy_n)_w:{\mathbb X}(w,w)\rto{\mathbb X}(w,v)$. Consequently,
$$\delta(1_w,m,n)=(\sy_m)_w(1_w,n)=(\sy_n)_w(1_w,n)=\delta(1_w,n,n)\geq\sk,$$
which implies $m=n$ when $\mathbb X$ is separated.
\end{proof}

\begin{rem}
When $\mathbb X$ is actually a $\sV$-metric category, so that one has the usual Yoneda embedding $\tilde{\sy}:{\mathbb X}\to[{\mathbb X},{\bf Met}_{\sV}]$, one easily sees that the diagram
$$\bfig\Atriangle/<-`->`->/<500,300>[ {[{\mathbb X}^{\op},{\bf Met}_{\sV}]}`{\mathbb X}`{[{\mathbb X}^{\op},{\bf Dist}_{\sV}]};\tilde{\sy}``\sy]\efig$$
commutes; here, the unnamed functor is induced by $(\text{-})_*:{\bf Met}_{\sV}\to{\bf Dist}_{\sV}$ of Section 6. Consequently, $\tilde{\sy}$ is, like {\sy}, an isometry. But note that, unlike $\tilde{\sy}$, {\em the isometry $\sy$ fails to be full}, that is: in general, its hom-maps do not map surjectively.
\end{rem}

Let us return to one of our principal goals and exploit the Theorem for the adjunction described in Proposition \ref{topological}:

\begin{cor}\label{main cor}
The unit ${\mathbb X}\to{\sf Path}{\mathbb X}$ of the right-adjoint functor ${\bf Met}_{\sV}\text{-}\Cat\to{\bf Metag}_{\sV}$ at the $\sV$-metagory $\mathbb X$ is an isometry.
\end{cor}

\begin{proof}
The isometry $\sy$ factors uniquely through the unit, which therefore must be an isometry as well:
$$\bfig\Atriangle/<-`->`->/<500,300>[{\sf Path}{\mathbb X}`{\mathbb X}`{[{\mathbb X}^{\op},{\bf Dist}_{\sV}].};``\sy]\efig$$
\end{proof}
\begin{exmp}(See \cite{AlioucheSimpson2017}, Section 11.)
As in Example \ref{2Met exmp}(1), consider ${\mathbb R}^n$ as a G\"{a}hler- 2-metric space  \cite{Gahler1963}. Via Example \ref{chaotic metag}, ${\mathbb R}^n$ becomes an ${\mathbb R}_+$-metagory, for which one can form the ${\mathbb R}_+$-metric category ${\sf Path}{\mathbb R}^n$. Its morphisms are polygonal paths in ${\mathbb R}^n$ and, as described in fair detail in \cite{AlioucheSimpson2017}, the
 distance of two morphisms $f, g$ in the same hom-set is computed as the infimum of the areas of all ``triangulated zero-volume bodies" that have the closed path following first $f$, and then $g$ in reverse direction, as their boundary.
\end{exmp}

\section{Sufficient conditions for transitivity}
Let us refine the claim made in Remark \ref{trans remarks}(1) and give a chain of increasingly weakening conditions for transitivity. We use $<\!\!<$ to denote the {\em totally below} relation in the complete lattice $\sV$, that is:
$$u<\!\!<v:\Longleftrightarrow\forall A\subseteq \sV\,(v\leq\bv A\Longrightarrow\exists a\in A\,(u\leq a)).$$
We will use this relation only for $v=\sk$ the $\otimes$-neutral element of $\sV$, and we then use $\varepsilon$ instead of $u$ since $\varepsilon<\!\!<\sk$ amounts to $\varepsilon>0$ in case $\sV={\mathbb R}_+$. Note that, for the bottom element $\bot$ in $\sV$, one has $\bot<\!\!<\sk$ precisely when $\bot<\sk$, that is, when $\sV$ is not the trivial quantale $\sf 1$.
\begin{prop}\label{trans prop}
For a $\sV$-metagory $\mathbb X$, consider the following conditions:
\begin{itemize}
\item[{\em (i)}] $\mathbb X$ is (induced by) a $\sV$-metric category;
\item[{\em (ii)}] for all $f:x\rightarrow y, \,g:y\rightarrow z$ in $\mathbb X$, there is $a:x\rightarrow z$ in $\mathbb X$ with $\sk\leq\delta(f,g,a)$;
\item[{\em (iii)}] for all $f:x\rightarrow y,\, g:y\rightarrow z$ in $\mathbb X$,  $\sk\leq\bv_{a:x\rightarrow z}\delta(f,g,a)\otimes\delta(f,g,a)$;
\item[{\em (iv)}] $\mathbb X$ is transitive.
\end{itemize}
Then $\rm{(i)}\Longrightarrow\rm{(ii)}\Lra\rm{(iii)}\Lra\rm{(iv)}$, and $\rm{(i)}$ and $\rm{(ii)}$ are equivalent when $\mathbb X$ is separated. Furthermore, if $\sV$ satisfies $\bv\{\varepsilon\,|\,\varepsilon<\!\!<\sk\}=\sk$, then {\em (iii)} is equivalent to
\begin{itemize}
\item[{\em(iii$'$)}] for all $\varepsilon\!<\!\!<\!\sk$ in $\sV$, $f\!:\!x\rightarrow y, g\!:\!y\rightarrow z$ in $\mathbb X$, there is $a\!:\!x\rightarrow z$ in $\mathbb X$ with $\varepsilon\!\leq\!\delta(f,g,a)\otimes\delta(f,g,a)$.
\end{itemize}
\end{prop}
\begin{proof}
 (i)$\Lra$(ii): Consider $a=g\cdot f$.  (ii)$\Lra$(iii): Trivially, from $\sk\leq\delta(f,g,a)$ one obtains
 $$\sk=\sk\otimes\sk\leq\delta(f,g,a)\otimes\delta(f,g,a)\leq \bv_{a':x\rightarrow z}\delta(f,g,a')\otimes\delta(f,g,a').$$ 
 (iii)$\Lra$(iv): Given $f,g,h,b,c$ as in Definition \ref{trans def}, for all $a:x\rightarrow y$ the left associativity law implies $\delta(f,g,a)\otimes\delta(f,g,a)\otimes\delta(g,h,b)\otimes\delta(f,b,c)\leq\delta(f,g,a)\otimes\delta(a,h,c)$. Taking joins on both sides, with (iii) this gives the left transitivity law, with the right transitivity law following similarly:
 $$\delta(g,h,b)\otimes\delta(f,b,c)=\sk\otimes \delta(g,h,b)\otimes\delta(f,b,c)\leq\bv_{a:x\rightarrow z}\delta(f,g,a)\otimes\delta(a,h,c).$$
(ii)$\Lra$(i): When $\mathbb X$ is separated, given $f:x\rightarrow y, g:y\rightarrow z$, there can only be one $a:x\rightarrow z$ with $\sk\leq\delta(f,g,a)$: if there is also $a'$, consider the tetrahedron $(f,g,1_z;a,g,a')$ to obtain $\sk\leq\delta(a,1_z,a')$, whence $a=a'$; then, with $g\cdot f:=a$, it is also routine to show that the $\sV$-metric $d$ induced by $\delta$ satisfies the condition of Definition \ref{metric cat}. (iii)$\Lra$(iii$')$ follows from the definition of $<\!\!<$, and with the hypothesis $\bv\{\varepsilon\,|\,\varepsilon<\!\!<\sk\}=\sk$, (iii$')\Lra$(iii) follows trivially. 
  \end{proof}
The Proposition suggests the study of those of $\sV$-metagories that allow for ``composition up to $\varepsilon$". Hence, adapting the terminology used in \cite{AlioucheSimpson2017}, we define:  
\begin{defn}
 For $\varepsilon\in\sV$, a $\sV$-metagory $\mathbb X$ is $\varepsilon$-{\em categorical} if for all 
 $f:x\rightarrow y, \,g:y\rightarrow z$ in $\mathbb X$, there is $a:x\rightarrow z$ in $\mathbb X$ with $\varepsilon\leq\delta(f,g,a)$. We denote the corresponding full subcategory of ${\bf Metag}_{\sV}$ by ${\varepsilon}\text{-}{\bf Metag}_{\sV}$. 
 We let $\bf{TransMetag}_{\sV}$ denote the full subcategory of transitive $\sV$-metagories.
 \end{defn}
\begin{rem}
(1) By Proposition \ref{trans prop}, the separated $\sk$-categorical $\sV$-metagories are precisely those that are induced by separated $\sV$-metric categories.

(2) The $\bot$-categorical metagories $\mathbb X$ are those satisfying the condition $${\mathbb X}(x,y)\neq\emptyset,\;{\mathbb X}(y,z)\neq\emptyset\;\Lra\;{\mathbb X}(x,z)\neq\emptyset$$
for all $x,y,z\in{\rm{ob}}{\mathbb X}$, which is used as the general hypothesis (7.1) in \cite{AlioucheSimpson2017}.
\end{rem}

With a slight adaptation of the argumentation given in Proposition \ref{trans prop}, one obtains the following inclusions between the three subcategories of ${\bf Metag}_{\sV}$ that seem to deserve further study:
\begin{thm}
Assume that the quantale $\sV$ satisfies $\bv\{\varepsilon\otimes\varepsilon\,|\,\varepsilon<\!\!<\sk\}=\sk.$ Then
$$\sk\text{-}{\bf Metag}_{\sV}\;\subseteq\;\bigcap_{\varepsilon<\!\!<\sk}\varepsilon\text{-}{\bf Metag}_{\sV}\;\subseteq\;{\bf TransMetag}_{\sV}.$$
For $\sV$-metagories $\mathbb X$ and $\mathbb Y$, if $\mathbb Y$ is $\sk$-categorical, so is $[{\mathbb X}, {\mathbb Y}]$. If both $\mathbb X$ and $\mathbb Y$ are $\sk$-categorical or transitive, ${\mathbb X}\otimes{\mathbb Y}$ has the respective property; likewise for the property of being $\varepsilon$-categorical for all $\varepsilon<\!\!<\sk$.
\end{thm}
\begin{proof}
The first inclusion is trivial since $\varepsilon<\!\!<\sk$ implies $\varepsilon\leq\sk$. For the second inclusion, one argues similarly as for (iii)$\Lra$(iv) of Proposition \ref{trans prop}. Indeed, for a $\sV$-metagory that is $\varepsilon$-categorical for all $\varepsilon<\!\!<\sk$, given $f,g,h,b,c$, one has $a_{\varepsilon}$ with $a_\varepsilon\leq\delta(f,g,a_{\varepsilon})$. With the left associativity law applied to the tetrahedron $(f,g,h,a_{\varepsilon},b,c)$ one obtains 
$\varepsilon\otimes\varepsilon\otimes\delta(g,h,b)\otimes\delta(f,b,c)\leq\delta(f,g,a_{\varepsilon})\otimes\delta(a_{\varepsilon},h,c).$ Taking joins on both sides we obtain the left transitivity law.

To show that the property of being $\sk$-categorical or transitive gets transferred from $\mathbb{X,Y}$ to ${\mathbb X}\otimes{\mathbb Y}$ is straightforward. To show the same for being $\varepsilon$-categorical for all $\varepsilon<\!\!<\sk$, first observe that, by hypothesis, for every such $\varepsilon$ we have some  $\eta<\!\!<\sk$ with $\varepsilon\leq\eta\otimes\eta$ 
(that is: $\sV$ has the {\em halving property}). Then, given $(f,f'):(x,x')\to(y,y'),\,(g,g'):(y,y')\to(z,z')$ in ${\mathbb X}\otimes{\mathbb Y}$, exploiting the hypothesis on $\mathbb{X,Y}$ at $\eta$ rather than $\varepsilon$, we can find $a:x\to z$ in $\mathbb X$, $a':x'\to z'$ in $\mathbb Y$ satisfying
$$\varepsilon\leq\eta\otimes\eta\leq\delta_{\mathbb X}(f,g,a)\otimes\delta_{\mathbb Y}(f',g',a')=\delta_{{\mathbb X}\otimes{\mathbb Y}}((f,f'),(g,g'),(a,a')).$$

The more surprising fact is that being $\sk$-categorical gets inherited by $[{\mathbb X},{\mathbb Y}]$ from $\mathbb Y$ (for which one actually does not need the halving property), as we show now. Given natural transformations $\alpha:F\to G,\,\beta:G\to H$, it is clear that, assuming a choice principle, for all $x\in{\rm{ob}}{\mathbb X}$ one can first find $\gamma_x:Fx\to Hx$ in $\mathbb Y$ with $\sk\leq\delta(\alpha_x,\beta_x,\gamma_x)$. Next, for every morphism $f:x\to y$ in $\mathbb X$ that is {\em not} an identity morphism, we can then choose $\gamma_f: Fx\to Hy$ in $\mathbb Y$ with $\sk\leq\delta(\gamma_x,Hf,\gamma_f)$; for $f=1_x$, with the convention of writing $\gamma_x$ for $\gamma_{1_x}$, this last inequality holds trivially. In this way $\gamma=(\gamma_f)_{f:x\rightarrow y}$ satisfies already one of the two required conditions to qualify as a natural transformation $F\to H$; once we have verified the other, the proof is complete since, by choice of $\gamma_x$, one trivially has $\sk\leq\delta(\alpha,\beta,\gamma)$. So, we are left with having to show $\sk\leq\delta(Ff,\gamma_y,\gamma_f)$, for all $f$.

To this end we observe first that the right associativity applied to $(\alpha_x,\beta_x,Hf;\gamma_x,\beta_f,\gamma_f)$ gives
$$\sk=\sk\otimes\sk\otimes\sk\leq\delta(\alpha_x,\beta_x,\gamma_x)\otimes\delta(\beta_x,Hf,\beta_f)\otimes\delta(\gamma_x,Hf,\gamma_f)\leq\delta(\alpha_x,\beta_f,\gamma_f).   $$
Next, exploiting the left associativity at the tetrahedron $(\alpha_x, Gf, \beta_y;\alpha_y,\beta_f,\gamma_f)$, we obtain
$$\delta(\alpha_x,\beta_f,\gamma_f)=\sk\otimes\sk\otimes\delta(\alpha_x,\beta_f,\gamma_f)\leq\delta(\alpha_x,Gf,\alpha_f)\otimes\delta(Gf,\beta_y,\beta_f)\otimes\delta(\alpha_x,\beta_f,\gamma_f)\leq\delta(\alpha_f,\beta_y,\gamma_f).$$
Finally, with the right associativity law applied to $(Ff,\alpha_y,\beta_y;\alpha_f,\gamma_y,\gamma_f)$ we have
$$\delta(\alpha_f,\beta_y,\gamma_f)=\sk\otimes\sk\otimes\delta(\alpha_f,\beta_y,\gamma_f)\leq\delta(Ff,\alpha_y,\alpha_f)\otimes\delta(\alpha_y,\beta_y,\gamma_y)\otimes\delta(\alpha_f,\beta_y,\gamma_f)\leq\delta(Ff,\gamma_y,\gamma_f),$$
as desired.
\end{proof}

\begin{rem}
Recall that the complete lattice $\sV$ is {\em (constructively) completely distributive} \cite{Wood2004, MonTop} if, for all $v\in\sV$, one has $\bv\{u\,|\,u<\!\!<v\}=v$. In this case, and in particular when $\sV$ is completely distributive in the classical sense \cite{Raney1960}, one trivially has $\bv\{\varepsilon\,|\,\varepsilon<\!\!<\sk\}=\sk.$ Flagg \cite{Flagg1997} gives sufficient conditions for the marginally stronger condition $\bv\{\varepsilon\otimes\varepsilon\,|\,\varepsilon<\!\!<\sk\}=\sk$ to hold; all quantales mentioned in Remark \ref{quantales} trivially satisfy (what we have termed above) the halving property.

As encountered elsewhere (see in particular \cite{Flagg1997, LaiTholen2017}), it is likely that further studies on $\varepsilon$-categorical $\sV$-metagories will require an in-depth analysis of additional properties of the quantale $\sV$, such as the complete distributivity of its underlying lattice.

\end{rem}

\section{Some remarks on 2-categorical metric approximation}

Considering that the category ${\bf Met}_{\sV}\text{-}\Cat$ is actually a 2-category, with the 2-cells given by natural transformations of $\sV$-contractive functors, it is natural to ask whether also ${\bf Metag}_{\sV}$ may be considered as a 2-category, with the 2-cells given by natural transformations of $\sV$-contractors. But clearly, the vertical composition of 2-cells in
${\bf Met}_{\sV}\text{-}\Cat$ takes advantage of the composition in the target category, which is not availabe for $\sV$-metagories. We therefore consider the following property of ${\bf Met}_{\sV}\text{-}\Cat$ (which is actually stronger than the property of being a 2-category),
as a more promising starting point for a potential generalization to $\sV$-metagories:

\begin{prop}\label{self-enriched}
${\bf Met}_{\sV}\text{-}\Cat$ is self-enriched, that is: it is enriched in the monoidal-closed category ${\bf Met}_{\sV}\text{-}\Cat$.
\end{prop}

\begin{proof}
With all other items following from well-known properties of natural transformations, the only one to be confirmed is that, for all $\sV$-metric categories $\mathbb{X, Y, Z}$, the horizontal composition
$$c_{\mathbb{X, Y, Z}}:[\mathbb X,\mathbb Y]\otimes[\mathbb Y,\mathbb Z]\to[\mathbb X,\mathbb Z],\quad(\varphi:F\rightarrow G,\,\psi:H\rightarrow J)\;\mapsto\; (\psi\circ \phi:HF\rightarrow JG),$$
$$\begin{tikzpicture}[>=stealth]
\path (0,0) node(x) {$\mathbb X$}
(1.85,0) node(y) {$\mathbb Y$}
(3.7,0) node(z) {$\mathbb Z$}
(0.92,1.0) node(F)  {$F$}
(2.8,1.0) node(H)  {$H$}
(2.8,-1.0) node(J)  {$J$}
(0.92,-1.0) node(G)  {$G$}
(1.2,0) node(psi)  {$\phi$}
(3.08,0) node(phi)  {$\psi$};
\draw [->] (0,.23) arc (150:30:30pt) ;
\draw [->] (1.85,.23) arc (150:30:30pt) ;
\draw [->] (0,-0.23) arc (-150:-30:30pt) ;
\draw [->] (1.85,-0.23) arc (-150:-30:30pt) ;
\draw[double distance=1.5pt,->] (.92,0.4) -- (.92,-0.4);
\draw[double distance=1.5pt,->] (2.8,0.4) -- (2.8,-0.4);
\end{tikzpicture}$$

$$\xymatrix{HF
\ar[r]^{H\phi}\ar[d]_{\psi F}\ar[rd]^{\psi\circ\phi} & HG
\ar[d]^{\psi G}\\JF\ar[r]_{J\phi} & JG}
$$

\noindent is $\sV$-contractive. But for $\varphi',\psi'$ running parallel to $\varphi,\psi$, respectively, the formulae of Remark \ref{VMetCat} give us:
\begin{align*}
d_{[\mathbb X,\mathbb Y]\otimes[\mathbb Y,\mathbb Z]}((\varphi,\psi),(\varphi',\psi')) &=d_{[\mathbb X,\mathbb Y]}(\varphi,\varphi')\otimes d_{[\mathbb Y,\mathbb Z]}(\psi,\psi')\\
&=  \bigwedge_{x\in{\rm ob}{\mathbb X}}d_{\mathbb Y}(\varphi_x,\varphi'_{x})\;\otimes \bigwedge_{y\in{\rm ob}{\mathbb Y}}d_{\mathbb Z}(\psi_y,\psi'_{y})\\
&\leq \bigwedge_{x\in{\rm ob}{\mathbb X}}d_{\mathbb Z}(J\varphi_x,J\varphi'_{x})\;\otimes \bigwedge_{x\in{\rm ob}{\mathbb X}}d_{\mathbb Z}(\psi_{Fx},\psi'_{Fx})\\
&\leq \bigwedge_{x\in{\rm ob}{\mathbb X}}d_{\mathbb Z}(J\varphi_x\cdot\psi_{Fx},\;J\varphi'_x\cdot\psi'_{Fx})\\
& \leq \bigwedge_{x\in{\rm ob}{\mathbb X}}d_{\mathbb Z}((\psi\circ\varphi)_x,(\psi'\circ\varphi')_x)\\
& = d_{[\mathbb X,\mathbb Z]}(\psi\circ\varphi,\psi'\circ\varphi').
\end{align*}
\end{proof}

As shown in Section 4, ${\bf Metag}_{\sV}$ is,
 like ${\bf Met}_{\sV}\text{-}\Cat$, symmetric monoidal closed. So, for all $\sV$-metagories ${\mathbb{X, Y}}$, we have the local $\sV$-metagories $[{\mathbb X},{\mathbb Y}]$ (Theorem \ref{main thm}), which provide us with a legitimate 
$\sV$-metagorical substitute for a vertical composition of natural transformations of $\sV$-contractors. Somewhat surprisingly, as shown in \cite{Tholen2019}, there is a genuine horizontal composition law. That is: for $\sV$-metagories ${\mathbb{X,Y,Z}}$, just like in ${\bf Met}_{\sV}\text{-}\Cat$,
there is a well-defined composition 
$$c_{\mathbb{X, Y, Z}}:[\mathbb X,\mathbb Y]\otimes[\mathbb Y,\mathbb Z]\to[\mathbb X,\mathbb Z],\quad(\varphi:F\rightarrow G,\,\psi:H\rightarrow J)\;\mapsto\; (\psi\circ \phi:HF\rightarrow JG),$$
which is easily seen to satisfy the associative law of enriched category theory; one just puts
$$(\psi\circ\varphi)_f:=\psi_{\varphi_{f}}:HFx\to JGy,$$
for all $f:x\to y$ in $\mathbb X$. Also, trivially, the transformations $1_{\rm{Id}_{\mathbb X}}\;(\mathbb X$ a $\sV$-metagory) provide the needed identity morphisms with respect to $\circ$ and, not at all trivially, the composition map $c_{\mathbb{X,Y,Z}}$ turns out to be a $\sV$-contractor. In summary then, one obtains the appropriate generalization of Proposition \ref{self-enriched}, as follows :
\begin{thm}{\em{\cite{Tholen2019}}}
$\bf{Metag}_{\sV}$ is self-enriched.
\end{thm}
We note that the added structure on $\bf{Metag}_{\sV}$ is compatible with the corresponding structure on ${\bf Met}_{\sV}\text{-}\Cat$; in particular: the induced $\sV$-area functor (see Proposition \ref{metag prop}) preserves the horizontal composition.
The theorem clearly offers the blueprint for a notion of $\sV$-{\em 2-metagory}, which is being pursued in \cite{Tholen2019}, and perhaps even for higher-dimensional approximate categorical structures, which, albeit  in unspecified terms, are alluded to in the Introduction of \cite{AlioucheSimpson2012}.
\\\\
{\em Acknowledgement:} We thank the anonymous referee for a report with numerous helpful thoughts and suggestions; in particular, for raising the question mentioned at the beginning of Section 9.
\\\\

\noindent REFERENCES

\end{document}